\documentclass[]{amsart}

\usepackage[leqno]{amsmath}
\usepackage{amssymb,amsfonts,xfrac,MnSymbol}
\usepackage{amsthm}
\usepackage{cite}
\usepackage{hyperref}
\usepackage[msc-links]{amsrefs}
\usepackage{todonotes}
\usepackage{enumitem}
\usepackage{graphicx}
\usepackage{tikz-cd, tikz}
\usepackage{color}
\usepackage{import}
\usepackage[capitalize]{cleveref}

\numberwithin{equation}{section}

\newtheorem{theorem}{Theorem}[section]
\newtheorem{claim}[theorem]{Claim}

\newtheorem{lemma}[theorem]{Lemma}
\newtheorem{corollary}[theorem]{Corollary}

\newtheorem*{theorem*}{Theorem}
\newtheorem*{claim*}{Claim}
\newtheorem*{proposition*}{Proposition}
\newtheorem*{lemma*}{Lemma}
\newtheorem*{corollary*}{Corollary}

\theoremstyle{definition}

\newtheorem{remark}[theorem]{Remark}

\newtheorem{fact}[theorem]{Fact}

\newtheorem*{definition*}{Definition}
\newtheorem*{observation*}{Observation}
\newtheorem*{remark*}{Remark}
\newtheorem*{example*}{Example}
\newtheorem*{question*}{Question}
\newtheorem*{exercise*}{Exercise}
\newtheorem*{fact*}{Fact}
\newtheorem*{notation*}{Notation}

\newcommand{\bbH}{\mathbb{H}}

\newcommand{\bbN}{\mathbb{N}}

\newcommand{\bbR}{\mathbb{R}}

\newcommand{\bbZ}{\mathbb{Z}}

\newcommand{\calG}{\mathcal{G}}
\newcommand{\calH}{\mathcal{H}}

\newcommand{\calL}{\mathcal{L}}
\newcommand{\calM}{\mathcal{M}}

\newcommand{\immerse}{\looparrowright}

\newcommand{\ii}{^{-1}}

\newcommand{\gen}[1]{\left< #1 \right>}

\newcommand{\tild}[1]{\widetilde{#1}}

\DeclareMathOperator{\Fix}{Fix}

\DeclareMathOperator{\diam}{diam}

\DeclareMathOperator{\CC}{CC}
\DeclareMathOperator{\Area}{Area}
\DeclareMathOperator{\rk}{rk}

\DeclareMathOperator{\Isom}{Isom}
\DeclareMathOperator{\Hull}{Hull} 
\title[Ascending chains of free subgroups in certain 3-manifold groups] {Ascending chains of free subgroups in closed hyperbolic and graph 3-manifold groups}
\author{Edgar A. Bering IV}
\address{Department of Mathematics \& Statistics, San José State University, San Jose, California}
\email{edgar.bering@sjsu.edu}
\author{Nir Lazarovich}
\address{Department of Mathematics, Technion---Israel Institute of Technology, Haifa, Israel}
\email{lazarovich@technion.ac.il}
\thanks{EB was supported by the Azrieli foundation.}
\thanks{NL was supported by the Israel Science Foundation (grant no. 1562/19)}
\date{}

\begin{document}

\begin{abstract}
Takahasi and Higman independently proved that
any ascending chain of subgroups of constant rank in a free group must stabilize. Kapovich and Myasnikov gave a proof of this fact in the language of graphs and Stallings folds. Using profinite techniques, Shusterman extended this constant-rank ascending chain condition to limit groups, which include closed surface groups. Motivated by Kapovich and Myasnikov's proof we provide two new proofs of this ascending chain condition for closed surface groups, and establish the ascending chain condition for free subgroups of constant rank in fundamental groups of closed hyperbolic and graph 3-manifolds.

These results are now subsumed by the more general framework established in joint work with Heikamp, Kohav, and Munro which both corrected a mistake in a previous version of this paper and generalized it.
The proof for closed hyperbolic 3-manifolds and graph manifolds is preserved in this unpublished note for its direct, geometric approach, which remains valid and particularly transparent.
\end{abstract}

\maketitle

\section{Introduction}

When does a set of free subgroups of a group satisfy the ascending chain condition? If a group $G$ contains a a subgroup isomorphic to the free group $F(a, b)$ then the subgroups $\gen{a} < \gen{a,bab^{-1}} < \gen{a, bab^{-1}, b^2ab^{-2}} < \cdots$ are a proper ascending chain of unbounded rank. Imposing a bound on the rank of each subgroup under consideration removes this obvious chain and it is a classical result due independently to Takahasi~\cite{takahasi}*{Theorem 1} and Higman~\cite{higman}*{Lemma} that any ascending chain of subgroups of constant rank in a fixed free group must stabilize.

Kapovich and Myasnikov~\cite{kapovich-myasnikov}*{Theorem 14.1} give a proof of this ascending chain condition using finite graphs following Stallings. To sketch the argument, suppose $G$ is a fixed finite-rank free group and $H_1 \le H_2 \le H_3 \le \cdots \le G$ is an ascending chain of subgroups of constant rank $r$. Fix a finite graph $\Gamma$ such that $G = \pi_1(\Gamma)$. Its universal cover $\tilde{\Gamma}$ is a tree. Each $H_i$ acts on $\tilde{\Gamma}$ and has a minimal non-empty invariant subtree $\tilde{\Gamma}_i$ with quotient a finite graph $\Gamma_i$. The nesting of subgroups induces graph maps $\Gamma_i \to \Gamma_{i+1}$. One may assume without loss of generality (\cref{no-free-factors}) that for all $i$ the subgroup $H_i$ is not contained in a proper free factor of $H_{i+1}$. This implies that each map $\Gamma_i \to \Gamma_{i+1}$ is surjective. However, each $\Gamma_i$ must have at least $r$ edges. Therefore the inclusion $H_i \le H_{i+1}$ can only be proper a finite number of times and the chain stabilizes.

Inspired by this argument we establish an ascending chain condition for free groups of constant rank in closed hyperbolic 3-manifold groups.

\begin{theorem}\label{main result for hyperbolic}
Fundamental groups of closed hyperbolic 3-manifolds do not admit proper ascending chains of free subgroups of constant rank.
\end{theorem}

In the setting of surface groups, this constant-rank chain condition was proved by Shusterman using the theory of limit groups and profinite completions~\cite{shusterman}*{Theorem 1.3}. 
To motivate the proof of \cref{main result for hyperbolic} we give a new proof of this ascending chain condition in the closed surface setting. Our proof uses only tools from surface theory: Mumford's compactness criterion for the moduli space of hyperbolic metrics replaces edge-counting as the source of finiteness which forces any chain to stabilize, as detailed in \cref{no ascending in surfaces}.

Graph manifolds are 3-manifolds whose JSJ decomposition consists only of Seifert-fibered pieces. 

\begin{theorem}\label{main result for graph manifolds}
Fundamental groups of graph manifolds do not admit proper ascending chains of free subgroups of constant rank.
\end{theorem}

To prove \Cref{main result for graph manifolds}, we prove it first for Seifert-fibered pieces (\cref{no ascending in SFS}), and then use the following combination theorem. A graph of groups decomposition is \emph{$k$-acylindrical} if the stabilizer of any path of length $k$ is finite.

\begin{theorem}\label{main result for graph of groups}
If a group $G$ admits an $k$-acylindrical graph of groups decomposition over abelian edge groups such that no vertex groups admits a proper ascending chain of constant rank free groups, then $G$ does not admit a proper ascending chain of constant rank free subgroups.
\end{theorem}

We remark that surface groups admit a 2-acylindrical graph of groups decomposition over cyclic groups with free vertex groups. Therefore, \cref{main result for graph of groups} gives a second new proof for surface groups.
The proof of our combination result should be compared to Shusterman~\cite{shusterman}*{Proposition 3.2}.

We remark that \Cref{main result for graph manifolds,main result for hyperbolic} are now subsumed by the results of joint work with Heikamp, Kohav, and Munro \cite{bering2026ascending}. In that work a bounded rank ascending chain condition is proved for (not necessarily free) subgroups of compact 3-manifolds groups and locally quasiconvex subgroups of (relatively) hyperbolic groups.

\bigskip 

In general, subgroups of $\Isom^+(\bbH^3)$ do not satisfy the ascending chain condition on bounded rank free groups: Calegari and Dunfield produced a non-discrete free group of rank 6 in $\Isom^+(\bbH^3)$ that is conjugate to a proper subgroup of itself~\cite{calegari-dunfield}. The union of these nested conjugates has a proper ascending chain of free subgroups of rank 6. 

There is an alternative definition that excludes the ``obvious'' ascending chains in a free group: Consider only ascending chains of finitely generated free groups in which each inclusion is not into a proper free factor. Higman proved a locally free group (a group where every finitely generated subgroup is free) is countably free if and only if every ascending chain of finitely generated free groups with no proper free factor inclusions stabilizes~\cite{higman-almost-free}*{Theorem 1}.\footnote{Higman uses the shorthand $n$-subgroup for finitely generated subgroup in this paper~\cite{higman-almost-free}*{p. 285}.} Countable free groups are the only countable, countably free groups. Thus, Higman's criterion provides a very strong constraint on groups where every factor-avoiding chain of finitely generated free subgroups stabilizes: every locally free subgroup must be free. This is true for surface groups, and one can verify the factor-avoiding ascending chain condition for surface groups using an argument similar to Higman's. 

However, 3-manifold groups can contain non-free locally free subgroups, even in geometric settings. Kurosh~\cite{kurosh} proved that the group $K = \langle a, b, t | tat^{-1} = [a, b] \rangle$ has a locally free subgroup which is not free. Thus, by Higman's characterization, $K$ must contain a proper ascending chain of finitely generated free groups where no inclusion is into a proper free factor. One can construct a space $X$ with $\pi_1(X) = K$ by gluing together an infinite string of one-holed torii according to the presentation relation, by thickening $X$ one obtains a 3-manifold. Hyperbolic geometry cannot rule out such phenomena: Maskit~\cite{maskit}*{Chapter VIII.E.9} realized $K$ as a Kleinian group. Building on Maskit, Anderson~\cite{anderson} used $K$ to construct infinitely many non-commensurable finite-volume hyperbolic 3-manifolds with $K$ as a subgroup of the fundamental group.

\subsection*{Acknowledgements}

We are grateful to Michah Sageev, particularly in the early stages of this work, Ian Biringer, David Futer, Chris Hrushka, and Henry Wilton, for helpful conversations. We are especially grateful to an anonymous referee who found the mistake in the previous version of this paper, and to Jakub Heikamp, Jack Kohav and Zachary Munro for our joint work fixing the mistake and and generalizing the techniques in the earlier version~\cite{bering2026ascending}.

\section{Avoiding splitting factors}

In each of our proofs the first essential reduction is from general chains to chains where $H_i$ is not contained in a proper free factor of $H_{i+1}$. This condition is used to ensure that a corresponding map of topological objects is essential. A version of lemma exists, in different language, as the first part of Kapovich and Myasnikov's proof, but it is not stated explicitly~\cite{kapovich-myasnikov}*{Proof of Theorem 14.1}.

\begin{lemma}\label{no-free-factors}
Let $G$ be a group. Suppose $H_1 \le H_2 \le \ldots \le G$ is an ascending chain of free subgroups of constant rank. Then, there exists $K_1\le K_2 \le \ldots \le G$ an ascending chain of free subgroups of constant rank such that for all $i$ the group $K_i$ is not contained in a proper subgroup $L \le K_{i+1}$ where $\rk(L) < \rk(K_i)$. In particular, $K_i$ is not contained in a proper free factor of $K_{i+1}$ for all $i$.
Moreover, if the chain $H_i$ does not stabilize, then the chain $K_i$ does not stabilize.
\end{lemma}
\begin{proof}
We prove the lemma by induction on the constant rank $r$ of $H_i$. If $r=1$, there is nothing to prove. 
Let $I\subset \bbN$ be the subset of indices such that $H_i$ is contained in a subgroup of $H_{i+1}$ with rank less than $r$. If $I$ is finite, then the tail $K_i = H_{N+i}$ where $N=\max I$ has
the required property. If $I$ is infinite, then for each $i\in I$ we get $H_i \le  K_i \le H_{i+1}$ where $K_i$ is a proper subgroup of $H_{i+1}$ and $\rk(K_i) < \rk(H_{i+1}) = r$. Thus, we get an ascending chain $\{K_i\}_{i\in I}$ of strictly smaller rank where each inclusion is proper. By the pigeonhole principle there is a subsequence $\{K_j\}_J \subseteq\{K_i\}_I$ with constant rank strictly smaller than $r$. Applying the induction hypothesis to $\{K_j\}_J$ completes the proof.
\end{proof}

\begin{lemma}\label{base not in free factor}
Let $G$ be a group. Suppose $H_1 \le H_2\le \ldots \le G$ is an ascending chain of free subgroups
and for all $i$ the group $H_i$ is not contained in a proper free factor of $H_{i+1}$. Then for all $i$ the group $H_1$ is not contained in a proper free factor of $H_i$.
\end{lemma}

\begin{proof}
We proceed by induction on $i$. When $i=1, 2$ the claim is true by hypothesis. So suppose $H_1$ is not contained in a proper free factor of $H_{i-1}$. Suppose $H_i = K\ast L$ is a non-trivial free splitting of $H_i$. By hypothesis $H_{i-1}$ is not contained in $K$ or $L$, so  every factor in the Kurosh decomposition of $H_{i-1} = F \ast K_{i-1} \ast L_{i-1}$ is nontrivial; here where $F$ is free, and $K_{i-1}, L_{i-1}$ are free products of cojugates of subgroups of $K$ and $L$ respectively. By the induction hypothesis $H_1$ is not contained in $F$, $K_{i-1}$, or $L_{i-1}$, and therefore $H_1$ is not contained in any conjugate of $K$ or $L$.
\end{proof}

We assume that the reader is familiar with Bass-Serre Theory, and use the definitions and notation of  Serre~\cite{serre-trees}.
Given a graph of groups $\calG = (\Gamma, \{G_v\}_{v\in V\Gamma}, \{G_e\}_{e\in E\Gamma})$ with Bass-Serre tree $T$, any finitely generated subgroup $H \le \pi_1(\calG) = G$ has a minimal invariant subtree in $T$. The induced action of a subgroup $H$ is \emph{elliptic} if the action has a global fixed point. If $H$ does not stabilize an edge then the minimal subtree is unique and denoted $T^H$. Let $\calG^H = (\Gamma^H, \{H_v\}_{v\in V\Gamma^H}, \{H_e\}_{e\in E\Gamma^H})$ denote the graph of groups corresponding to $T^H$ and denote its underlying graph $\Gamma^H$. 

When taking $H=G$, the underlying graph of $\Gamma^G$ is a subgraph of $\Gamma$ and the vertex and edge groups in $\calG^G$ are the same as in $\calG$. 
A graph of groups is \emph{minimal} if its associated tree is minimal, or equivalently if $\calG^G = \calG$. For any two nested subgroups with unique minimal subtrees, $K\le H \le \pi_1(\calG)$ the inclusion of minimal subtrees induces a map of underlying graphs $\Gamma^K \to \Gamma^H$. 

An ascending chain of subgroups of a fixed free group can be represented as a chain of graph maps. In this setting, if the chain of subgroups satisfies the conclusion of \cref{no-free-factors} then each representative graph map is surjective. When working with a chain of free subgroups in a group with an abelian splitting an analogous surjectivity can be imposed on the induced maps of graphs underlying the induced splitting of the chain.

\begin{lemma}\label{surjective graph of groups}
Suppose $\calG = (\Gamma, \{G_v\}_{v\in V\Gamma}, \{G_e\}_{e\in E\Gamma})$ is a $k$-acylindrical finite graph of groups with abelian edge groups with fundamental group $G = \pi_1(\calG)$. Suppose $H_1 \le H_2 \le \ldots \le G$ is an ascending chain of free subgroups of constant rank. Then,  there exists $K_1\le K_2 \le \ldots \le G$ an ascending chain of free subgroups of constant rank such that either each $K_i$ is elliptic or each $K_i$ is non-elliptic and the induced map $\Gamma^{K_i} \to \Gamma^{K_{i+1}}$ is surjective for all $i$.
Moreover, if the chain $H_i$ does not stabilize, then the chain $K_i$ does not stabilize.
\end{lemma}

\begin{proof}
First observe that for every non-elliptic free subgroup $H \le G$ all edge groups of $\calG^H$ are abelian and hence trivial or infinite cyclic.
For a non-elliptic finitely generated free subgroup $H\le G$ define its complexity to be 
\[c(H) = \beta_1(\Gamma^{H}) + \rk(H) + E'(H)\in \bbN,\]
where $\beta_1$ is the first betti number and $E'(H)$ is the number of edges of $\calG^H$ with non-trivial stabilizer. For an elliptic finitely generated free group $H\le G$ define $c(H) = \rk(H)$. 
Let $\Lambda$ be a connected component of $\Gamma^H \setminus e$ for some edge $e$.
Let $\calL$ be the restriction of $\calG^H$ to $\Lambda$, i.e. $\calL$ is the graph of groups whose underlying graph is $\Lambda$, and whose vertex and edge groups are the same as that of $\calG$ (on $\Lambda$). Finally, let $L = \pi_1(\calL)$ and note that $L\le G$, and that the associated tree of $\calL$ is a subtree of the one associated to $\calG$.

\begin{claim}\label{cyclic complexity claim}
$c(L) < c(H)$.
\end{claim}

\let\oldqed\qedsymbol
\renewcommand{\qedsymbol}{$\lozenge$}
\begin{proof}[Proof of claim.]
Note that $\calL$ may not be minimal, but a minimal graph of groups $\calG^L$ can be obtained by restricting to a (necessarily finite) subgraph $\Gamma^L$ of $\Lambda$.
It is immediate that $\beta_1(\Gamma^L) \le \beta_1(\Lambda) \le \beta_1(\Gamma^H)$ and $E'(\Gamma^L)\le E'(\Lambda)\le E'(\Gamma^H)$. 
By Grushko's theorem~\cite{grushko} if $H_e=1$ then $\rk(L)<\rk(H)$, and so $c(L)<c(H)$.
Otherwise, $E'(L) < E'(H)$ and it follows from either Shenitzer's theorem on cyclic splittings~\cite{shenitzer} of free groups or Swarup's theorem on cyclic HNN extensions of free groups~\cite{swarup} that $\rk(L)\le \rk(H)$. So again $c(L) < c(H)$.\footnote{Stallings gives a unified account of the two splitting theorems used here~\cite{stallings-folding-trees}.}
\end{proof}

\begin{claim}\label{claim: bound on complexity}
    There exists $C_k$ so that if $\rk(H)=k$ then $c(H)\le C_k$.
\end{claim}

\begin{proof}[Proof of claim.]
First observe that 
$\beta_1(\Gamma^H)\le \rk(H)$. It remains to bound $E'(H)$. This is a form of acylindrical accessibility (cf. Delzant \cite{delzant} or Sela \cite{sela}). 
Bestvina and Feighn establish a bound $\gamma(H)$ on the number of edges in a reduced graph of groups with small edge groups for $H$~\cite{BF}. 
Note that the edge groups of $\calG^H$ are trivial or infinite cyclic, and so they are \emph{small}---that is, every action of an edge group of $\calG^H$ on a tree stabilizes a point, an end, or a pair of ends.
The induced splitting $\calG^H$ of $H$ might not be \emph{reduced}---that is, it may include a degree two vertex incident to two distinct edges and whose vertex group is equal to one of the adjacent edge groups; this edge is \emph{redundant}. 
In this case, one can contract the redundant edge and obtain a smaller graph of groups with the same fundamental group.
A path of redundant edges with non-trivial stabilizer can be of length at most the acylindricity constant $k$. 
Thus $E'(H)\le k \cdot \gamma(H)$.
\end{proof}
\renewcommand{\qedsymbol}{\oldqed}

We now prove the lemma by induction on the maximum complexity $c = \max c(H_i)$ of the ascending chain $H_1\le H_2 \le\dots $ (note that this maximum is attained by \cref{claim: bound on complexity}).
 If $c=1$ then each $H_i$ is an elliptic cyclic subgroup, so take $K_i = H_i$.

Let $I$ be the subset of indices such that the induced map $f_i : \Gamma^{H_i} \to \Gamma^{H_{i+1}}$ is not surjective. If $I$ is finite, then the tail $K_i = H_{N+i}$ where $N = \max I$ has the required property. If $I$ is infinite then for each $i\in I$ we define a subgroup $H_i < K_i < H_{i+1}$ with $c(K_i) < c(H_i)$ as follows. Let $\Lambda_i = f_i(\Gamma^{H_i})$, $\calL_i = \calG^{H_{i+1}}|_{\Lambda_i}$ and set $K_i = \pi_1(\calL_i)$. The graph $\Lambda_i$ can be obtained from $\Gamma^{H_{i+1}}$ by deleting a finite list of edges $e_1,\ldots, e_k$ in sequence. Let $\Lambda_i^j$
be the connected component of the intermediate graph where $e_1,\ldots, e_j$ have been deleted that contains $\Lambda_i$, and set $L_j = \pi_1(\Lambda_i^j)$. Either $L_j = L_{j+1}$, or $\Lambda_i^j$ is minimal for $L_j$ and by \cref{cyclic complexity claim} $c(L_{j+1}) < c(L_j)$. Moreover by \cref{cyclic complexity claim}, since $\Gamma^{H_{i+1}}$ is minimal $c(L_1) < c(H_{i+1})$.  We conclude $c(K_i) < c(H_{i+1}) \le c$. By the pigeonhole principle there is a subsequence $\{K_j\}_J \subseteq \{K_i\}_I$ of constant rank, and either every $K_j$ is elliptic or non-elliptic. In the first case we are done, in the second case applying the induction hypothesis to $\{ K_j\}$ completes the proof.
\end{proof}

\begin{proof}[Proof of \cref{main result for graph of groups}]
Let $\calG = (\Gamma, \{G_v\}_{v\in V\Gamma}, \{G_e\}_{e\in E\Gamma})$ be a finite graph of groups with abelian edge groups. Let $G=\pi_1(\calG)$. Let $H_1\le H_2 \le \ldots$ be an  ascending chain of free subgroups of constant rank.

By \cref{surjective graph of groups}, we may assume without loss of generality that either all $H_i$ are elliptic, or, all $H_i$ are non-elliptic and the induced graph maps $\Gamma^{H_i}\to \Gamma^{H_{i+1}}$ are surjective. We now analyze each of the cases.

If $H_i$ are elliptic. Consider the fixed point set $\Fix(H_i)$ of $H_i$. By acylindricity, $\Fix(H_i)$ has finite diameter. Since $H_i \le H_{i+1}$, $\Fix(H_i) \supseteq \Fix(H_{i+1})$ and so $\bigcap_i \Fix(H_i) \ne \emptyset$. Therefore, $H_i$ stabilize the same vertex and $H_i$ is an ascending chain in a conjugate of some $G_v$; by hypothesis the chain stabilizes.

In the second case, since each graph map $\Gamma^{H_i}\to \Gamma^{H_{i+1}}$ is surjective, eventually the sequence of graphs stabilizes to a graph $\Lambda$. Without loss of generality assume $\Gamma^{H_i} = \Lambda$ for all $i$ and that the morphism $\calH_i \to \calG$ induces a fixed map $f:\Lambda\to \Gamma$ on the common underlying graph. Let $T$ be the Bass-Serre tree for $\Gamma$ and $T^{H_1}$ the minimal subtree for $H_1$. Since $T^{H_1}\subseteq T^{H_i}$ for all $i$, for each vertex $v\in T^{H_1}$, there is a chain $(H_i)_v$ of stabilizers. By assumption each of the chains $(H_i)_v$ eventually stabilizes. Since $T^{H_1}\subseteq T^{H_i}$, up to conjugacy the vertex stabilizers in each $T^{H_i}$ are determined by the finitely many $H_1$-orbits of vertices in $T^{H_1}$, since each tree has the same finite graph quotient $\Lambda$. Thus, at some finite step $n$, the graphs of groups $\Gamma^{H_i}$ stabilize, which implies $H_i$ stabilizes, as required.
\end{proof}

\section{Hyperbolic isometries and total displacement}

In this section we record some notions from hyperbolic geometry, which are common to our surface and hyperbolic 3-manifold arguments.

Given a subset $X \subseteq \bbH^n \cup S_\infty^{n-1}$ its convex hull is denoted $\Hull(X)$. For $g\in \Isom(\bbH^n)$ we let $\tau(g)$ denote the minimum translation length; recall $\tau(g)$ is a conjugacy invariant. An isometry $g\in\Isom(\bbH^n)$ is \emph{parabolic} if $\tau(g) = 0$ but $g$ fixes no point of $\bbH^n$. A subgroup $\Isom(\bbH^n)$ is parabolic if all of its elements are parabolic. When $G\le \Isom(\bbH^n)$ is discrete, torsion-free, and non-abelian we denote the limit set by $\Lambda G \subseteq S^{n-1}_\infty$ and the \emph{convex core} by $\CC(G) = \Hull(\Lambda G) / G$.
Suppose $A \subset \Isom^+(\bbH^n)$ is a finite set of isometries.
Define the \emph{minimum total displacement} $\tau(A)$ to be
\[ \tau(A) = \inf_{x\in\bbH^n} \sum_{a\in A} d(x, a.x). \]
Note that if $\gen{A}$ is not parabolic then $\tau(A) > 0$. This agrees with the previous definition when $A$ is a singleton. For a finitely generated group of isometries we define the minimum total displacement of $G$ to be the infimum over that of finite generating sets
\[ \tau(G) = \inf\{ \tau(A)\,|\,A\text{ is a finite generating set of }G\} \] 
and, as a short hand, if $M$ is a hyperbolic $n$-manifold so that $M=\bbH^n/G$ for a discrete subgroup $\pi_1(M)\simeq G\le \Isom(\bbH^n)$, we write $\tau(M) = \tau(G)$.
Since $M$ determines $G$ uniquely up to conjugation, and $\tau$ is conjugation invariant, $\tau(M)$ is well-defined.

\begin{lemma}\label{conjugacy finiteness}
Suppose $M$ is a closed hyperbolic $n$-manifold. Given $C \in \bbR$ there are finitely many conjugacy classes of subsets $A\le \pi_1(M) < \Isom^+(\bbH^n)$ such that $\tau(A) \le C$.
\end{lemma}

\begin{proof}
Fix a base point $o\in \bbH^n$.
Let $A \subset \pi_1(M)$ be a set of elements such that $\tau(A) \le C$. 
Let $x\in \bbH^n$ be the point realizing $\tau(A)$. 
There is some element $h\in \pi_1(M)$ such that that $d(h.x, o) \le \diam(M) = D$. Conjugating by $h$,
for all $a\in A^h$ we have $d(o, a.o) \le 2D + C$. The action of $\pi_1(M)$ is properly discontinuous, so we conclude $A^h$ is contained in the finite set $\{ g\in \pi_1(M) \;|\; d(o, go) \le 2D + C\}$.
\end{proof}

In the sequel, we will apply \cref{conjugacy finiteness} in combination with the following theorem of Ohshika and Potyagailo (in light of the tameness theorem~\cites{agol, calegari-gabai} we remove the topologically tame hypothesis in our quotation).

\begin{theorem}[\cite{ohshika-potyagailo}*{Theorem 1.3(b)}]\label{conjugate equal}
Suppose $G \le \Isom^+(\bbH^3)$ is a discrete, non-elementary, torsion-free finitely generated group. Then for all $\alpha \in \Isom^+(\bbH^3)$, the inclusion $\alpha G \alpha^{-1} \subset G$ implies $\alpha G \alpha^{-1} = G$.
\end{theorem}

Note that this implies the analogous statement for subgroups of $\Isom^+(\bbH^2)$, originally due to Huber~\cite{huber}.

\section{Surface groups}

Our primary tool for working with surface groups will be various moduli spaces of hyperbolic metrics.  For an oriented surface $\Sigma$ of genus $g$ with $b$ boundary circles we denote its \emph{moduli space} of complete, finite-area hyperbolic metrics with totally geodesic boundary $\calM_{g,b}$ or $\calM(\Sigma)$; when $b=0$ we omit the index. We need the following consequence~\cite{primer}*{\S12.4.3} of Mumford's compactness criterion~\cite{mumford} in order to prove that certain collections of surfaces land in a compact part of a fixed moduli space.
\begin{fact} \label{compactness criterion}
Suppose $\Sigma$ is a finite-type surface with $\chi(\Sigma) < 0$. For all
$\epsilon > 0$ the subset $\calM^\epsilon(\Sigma) \subset \calM(\Sigma)$ consisting of hyperbolic structures where all essential arcs and simple closed curves have length at least $\epsilon$ is compact.
\end{fact} 
\cref{compactness criterion} allows us to bound $\tau$. Indeed, for a fixed $a\in \pi_1(\Sigma)$, the function $\mathbb{H}^2\times \calM(\Sigma) \to \mathbb{R}$ induced by $d(x,a.x)$ is continuous. Thus $\tau$ is the infimum of a family of continuous functions on $\calM(S)$, namely $\tau(A)$ for a generating set $A$, and we conclude $\tau$ is bounded on any compact set. We will appeal to this in the sequel without further comment.

\begin{theorem}\label{no ascending in surfaces}
Let $S$ be a compact 2-orbifold with a finite surface cover and let $r\in\bbN$. Every ascending chain of free groups of constant rank in $\pi_1(S)$ stabilizes.
\end{theorem}
\begin{proof}
The orbifold $S$ has a finite cover $S'$ which is an orientable surface. By taking intersections, an ascending chain of free groups of constant rank in $\pi_1(S)$ will give rise to an ascending chain of free groups of bounded rank in $\pi_1(S')$, which by the pigeonhole principle contains an ascending subsequence of free groups of constant rank. This reduces the proof to the case that $S$ is an orientable surface. Furthermore, if $S$ has boundary, then $\pi_1(S)$ is free, and the result follows from Takahasi or Higman~\cites{takahasi,higman}. If $S$ has genus zero or one there is nothing to prove. Thus, we may assume that $S$ is a closed orientable hyperbolic surface.

Let $G=\pi_1(S)$, and let $H_1\le H_2 \le \ldots$ be an ascending chain of free subgroups of rank $r$ in $G$. By \cref{no-free-factors} assume further that each $H_i$ is not contained in a free factor of $H_{i+1}$. If $r=1$ let $H_i = \langle \alpha_i\rangle$. Since $\tau(\alpha_{i+1}) \le \tau(\alpha_i)$ and translation lengths are bounded below by the girth of $S$ the chain must stabilize. 

Now suppose $r\geq 2$. Note that each $H_i$ is a discrete group of isometries of $\bbH^2$. Let $S_i = \CC(H_i)$. Since the rank of $H_i$ is fixed, by the pigeonhole principle, after passing to a subsequence we may assume that all $S_i$ are homeomorphic to a fixed surface $\Sigma$.
Since $H_1\le H_i$ we have $\Lambda H_1\subseteq \Lambda H_i$; after passing to the quotient this produces an isometric immersion $f_i: S_1 \immerse S_{i}$.
Let $\alpha(S_i)$ be the length of the shortest essential arc $\gamma_i$ in $S_i$. 
By \cref{base not in free factor} $H_1$ is not contained in a free factor of $H_i$, 
and in particular every map $f'_i:S_1\to S_i$ homotopic to $f_i$ must satisfy $f_i'(S_1)\cap \gamma_i \ne \emptyset$.
Therefore, the pre-image $f_i^{-1}(\gamma_i)$ contains an essential arc of $S_1$ and we conclude $\alpha(S_i) \ge \alpha(S_1) >0$.
Moreover, the length of the shortest essential curve on each $S_i$ is bounded below by the length of the shortest curve of $S$.

We conclude that the hyperbolic surfaces $S_i$ are contained in a compact subset of $\calM(\Sigma)$, 
so there is a uniform bound on $\tau(S_i)$. 
Let $A_i \subset H_i$ be a generating set realizing $\tau(S_i)$. By \cref{conjugacy finiteness} and the pigeonhole principle there exists an infinite set of indicies $J$ such that for all $i\le j\in J$ the sets $A_i$ and $A_j$ are conjugate in $\pi_1(S)$, and therefore $H_i$ is a conjugate of $H_j$ and a subgroup of $H_j$. Hence $H_i = H_j$ for all $i, j\in J$ by \cref{conjugate equal}, and the chain stabilizes.
\end{proof}

\begin{remark}\label{remark about surface subgroups}
Every subgroup of a surface group is either free or a closed surface group of strictly greater rank. Therefore, \cref{no ascending in surfaces} implies that every ascending chain of constant rank subgroups in a fixed orbifold group stabilizes. 
\end{remark}

\section{Graph manifolds}

Recall that a group is \emph{slender} if every subgroup is finitely generated. Equivalently, every ascending chain of subgroups stabilizes.

\begin{lemma}\label{extension lemma}
Let $1\to N \to G \to S \to 1$ be a short exact sequence of groups. If $N$ is slender and every ascending chain of constant rank subgroups of $S$ stabilizes then every ascending chain of constant rank subgroups of $G$ stabilizes.
\end{lemma}

\begin{proof}
Let $H_1\le H_2 \le \ldots \le G$ be an ascending chain of constant rank subgroups of $G$. The chain $H_i \cap N$ stabilizes because $N$ is slender, and so does the image chain $\bar{H}_i \le S$ by hypothesis. By the short five lemma, the chain $H_i$ also stabilizes. 
\end{proof}

\begin{corollary}\label{no ascending in SFS}
Let $M$ be a Seifert-fibered 3-manifold. Every ascending chain constant rank subgroups in $\pi_1(M)$ stabilizes. 
\end{corollary}

\begin{proof}
Let $M$ be a Seifert-fibered 3-manifold, and let $G=\pi_1(M)$. Thus, $G$ fits into an exact sequence
\[ 1\to N \to G \to S \to 1\] where $N\simeq \bbZ$ is the fiber group and $S$ is the fundamental group of a 2--orbifold with finite surface cover~\cite{Scott1983}*{\S1-3}. Note that $N\simeq \bbZ$ is slender, and every ascending chain of constant rank subgroups of $S$ stabilizes by \cref{no ascending in surfaces,remark about surface subgroups}. The corollary follows from \cref{extension lemma}.
\end{proof}

\begin{remark}
    A similar proof works for 3-manifolds with Sol geometry as they are virtually $\bbZ^2$-by-$\bbZ$ groups.
\end{remark}

\begin{proof}[Proof of \Cref{main result for graph manifolds}]
Suppose $M$ is a closed graph manifold. As in the orbifold portion of the surface case we may pass to a finite-sheeted cover, so assume $M$ is orientable. 
Since $M$ is a graph manifold, $M$ decomposes along incompressible tori into Seiferet-fibered pieces. This decomposition splits $\pi_1(M)$ as an $2$-acylindrical graph of groups with abelian edge groups \cite{lafont}*{Proposition 8.2}. The theorem now follows from \cref{main result for graph of groups,no ascending in SFS}.
\end{proof}

\section{Hyperbolic 3-manifold groups}

\begin{proof}[Proof of \Cref{main result for hyperbolic}]
As in the surface setting by passing to the orientation double cover we may assume without loss that $M$ is orientable. Let $G=\pi_1(M)$, and let $H_1\le H_2 \le \ldots \le G$ be an ascending chain of free groups of rank $r$. By \cref{no-free-factors} assume further that each $H_i$ is not contained in a free factor of $H_{i+1}$.

Since each $H_i$ is finitely generated it is
a consequence of the tameness theorem for hyperbolic 3-manifolds that $H_i$ is either geometrically finite or a virtual fiber~\cites{agol, calegari-gabai}. 
Since each $H_i$ is free, it cannot be a virtual fiber (since virtual fibers are fundamental groups of closed surfaces, and hence not free). 
Therefore every $H_i$ is geometrically finite.

We will show that there is a uniform upper bound for $\tau(H_i)$. The theorem follows from this bound via the pigeonhole principle, \cref{conjugacy finiteness,conjugate equal} as in the surface case.

In order to bound $\tau(H_i)$ we consider two cases depending on the limit sets of $H_i$. 

If there is a round circle $C \subset \partial \bbH^3$ such that the limit sets $\Lambda H_i \subseteq C$ for all $i$, then each $H_i$ acts on a hyperbolic plane $\bbH^2 = \Hull(C)$ and each convex core $S_i = CC(H_i)$ is a hyperbolic surface. By the pigeonhole principle we may pass to a subsequence and assume all $S_i$ are homeomorphic to a fixed surface $\Sigma$. As in the proof of \cref{no ascending in surfaces} the length $\alpha(S_i)$ of the shortest essential arc in each $S_i$ is bounded below by $\alpha(S_1) > 0$. Since there is a lower bound on the length of each element of $G$, there is a lower bound on the girth of each $S_i$. Therefore the $S_i$ land in a compact subset of the moduli space of $\Sigma$, which implies there is a uniform upper bound bound on $\tau(H_i)$.

Otherwise, after passing to a tail sequence, no limit set $\Lambda H_i$ is contained in a round circle. Let $M_i = \CC(H_i)$ and consider the boundary surfaces $\partial M_i$. Each $\partial M_i$ is a closed connected hyperbolic surface in the intrinsic metric~\cite{epstein-marden}*{Theorem II.1.12.1}. Each $M_i$ is a homeomorphic to a handlebody of genus $r$.
Each hyperbolic surface $\partial M_i$ determines a point in the moduli space $\calM_{r}$. 
First we will show that the sequence $\{\partial M_i\}$ is contained in a compact subset of $\calM_{r}$.
 
For a contradiction suppose there exists a sequence $\gamma_i \subset \partial M_i$ of essential simple closed curves, geodesic in the intrinsic metric, such that the lengths $\ell(\gamma_i)\to 0$. The curves $\gamma_i$ must eventually be nullhomotopic in $M$. Indeed, for all $\gamma_i$, we have $\ell(\gamma_i) \ge \tau([\gamma_i])$ which is bounded below for non-identity elements of $G$. 
Each $M_i$ is a quotient of a convex hull in $\bbH^3$, so $\partial M_i$ is mean convex. Thus, by a theorem of Meeks and Yau~\cite{meeks-yau}*{Theorem 10}, there exists a minimum area properly embedded disk $D_i\subset M_i$ with $\partial D_i = \gamma_i$. Further, by convexity, the isoperimetric inequality of $\bbH^3$ holds also for $M_i$, thus $\ell(\gamma_i)\to 0$ implies $\Area(D_i)\to 0$.

To arrive at a contradiction, we will show that $\Area(D_i)$ is bounded away from zero. Since $H_1 \le H_{i}$, we obtain an immersion $\rho_i: M_1 \immerse M_i$ as in the proof of \cref{no ascending in surfaces}. By \cref{base not in free factor} $H_1$ is not contained in a free factor of $H_i$, so $\rho_i$ cannot be homotoped away from $D_i$, i.e. for every map $\rho':M_1\to D_i$ homotopic to $\rho_i$, we have $\rho'(M_i)\cap D_i\ne \emptyset$.
By perturbing $\rho_i$ slightly if needed, we may assume that $\rho_i(\partial M_1)$ and $D_i$ are transverse. So $\partial M_1 \cap \rho\ii(D_i)$ is a disjoint collection of simple closed curves on $\partial M_1$. 
There must be at least one curve $\eta$ of $\partial M_1 \cap \rho\ii(D_i)$ which is essential in $M_1$ (as otherwise, $\rho_i(M_1)$ could be homotoped away from $D_i$). 
Its image $\eta':=\rho_i(\eta)$ bounds a disk $E'$ in $D_i$. 
The disk $E'$ is necessarily a minimal filling disk for $\eta'$ since $D_i$ was minimal.
It follows that the lift $\tild E'$ of $E'$ to the universal cover of $M_i$ must lie in the convex hull of $\tild \eta'$ (as otherwise, nearest point projection of $\tild E'$ will be a smaller area disk).
Therefore, there exists a disk $E$ filling $\eta$ in $M_1$ such that $\rho_i(E)=E'$ and since $\rho_i$ is a local isometry $\Area(E) =\Area(E')\leq \Area(D_i)$. 
The minimum area of a compression disk for $M_1$ is positive and a lower bound for $\Area(D_i)$. This is a contradiction; thus the systole of $\partial M_i$ is uniformly bounded away from $0$ for all $i$; that is $\partial M_i$ is contained in a compact subset of $\calM_{r}$.

The inclusion $\partial M_i \to M_i$ induces a surjection $f_i : \pi_1(\partial M_i) \to H_i$.
For each $\alpha\in\pi_1(\partial M_i)$, and $x\in\partial\Hull(\Lambda H_i) = \partial\widetilde{ M}_i$, convexity implies $d_{\bbH^3}(f_i(\alpha).x, x) \leq d_{\partial \widetilde{M}_i}(\alpha.x, x)\le d_{\widetilde{\partial M_i}}(\alpha.x,x)$.
Thus $\tau(H_i) \le \tau(\partial M_i)$. The latter is uniformly bounded above for all $i$ since $\partial M_i$ is contained in a compact subset of $\calM_{r}$.
\end{proof}

\bibliographystyle{plain}

\end{document}